\documentclass[a4paper, 11pt,reqno]{amsart}
\title[QUP for Gabor Transform]{Qualitative Uncertainty Principle for \linebreak Gabor Transform on Certain Locally \linebreak Compact Groups}
\usepackage[ansinew]{inputenc}
\usepackage{yfonts}
\usepackage{amscd}
\usepackage{amssymb, latexsym}
\usepackage{mathrsfs}
\usepackage{fancyhdr}
\usepackage{enumerate}
\usepackage{mathtools}
\usepackage{setspace}
\usepackage{latexsym}
\newtheorem{thm}{Theorem}[section]
\newtheorem{rem}[thm]{Remark}

\newtheorem{lem}[thm]{Lemma}
\newtheorem{prop}[thm]{Proposition}

\allowdisplaybreaks
\author{JYOTI SHARMA}
\address{Department of Mathematics, University of Delhi,  Delhi-110007, India}
\email{jsharma3698@gmail.com}

\author{AJAY KUMAR}
\address{Department of Mathematics, University of Delhi, Delhi-110007, India}
\email{akumar@maths.du.ac.in}


\begin{document}
\begin{abstract}
 Classes of locally compact groups having qualitative uncertainty principle for Gabor transform have been investigated. These include Moore groups, Heisenberg Group $\mathbb{H}_n, \mathbb{H}_{n} \times D,$ where $D$ is discrete group and other low dimensional  nilpotent Lie  groups.
\end{abstract}
\maketitle

\section{Introduction}
\noindent For a  second countable, unimodular group  G of type I,  let $ m $ be the left Haar measure and $\mu$  the Plancherel measure on the dual object $\widehat{G}.$ Let $H$ be closed normal subgroup  of $G$, we identify a representation $\pi$ of $G/H$ with its pullback to $G$, so  $\widehat{G/H}$ will always considered as a closed subset of $\widehat{G}$.  Let 
\[
\mathcal{H}_{(x,\pi)}= \pi(x)\text{HS}(\mathcal{H}_{\pi})
\]
 where $\pi(x)\text{HS}(\mathcal{H}_{\pi})={\{\pi(x)T:T \in \text{HS}(H_{\pi})\}} $, $\mathcal{H}_{(x,\pi)}$ is a Hilbert Space with the inner product given by 
 \[ 
\left\langle \pi(x)T,\pi(x)S \right\rangle_{\mathcal{H}_{(x,\pi)}} = tr(S^*T) = \langle T,S\rangle_{\mathcal{H}_{(x,\pi)}}
 \] 
 Also, $\mathcal{H}_{(x,\pi)} = \text{HS}(\mathcal{H}_{\pi})$ for all $(x,\pi)\in G \times \widehat{G}.$ The family $\{\mathcal{ H}_{(x,\pi)}\}_{(x,\pi)\in G \times \widehat{G}}$ of Hilbert spaces is a field of Hilbert spaces over $G\times \widehat{G}$. Let $\mathcal{H}^2(G\times \widehat{G})$ denote the direct integral of $\{\mathcal{ H}_{(x,\pi)}\}_{(x,\pi)\in G \times\widehat{G}}$, that is, the space of all measurable vector Fields $F \text{ on } G\times\widehat{G}$ such that 
 \[  ||F||^2_{H^2(G\times \widehat{G})} = \int_{G\times\widehat{G}} || F(x,\pi)||^2_{(x,\pi)} dxd\pi < \infty.
 \]
  $H^2(G\times \widehat{G})$ is a Hilbert space with the inner product given by 
 \[
 \langle F,K \rangle_{H^2(G\times \widehat{G})} = \int_{G\times \widehat{G}} tr[F(x,\pi)K(x,\pi)^*]dxd\pi.
  \]
  
  \noindent For $f \in C_c(G)$, the set of all continuous functions on $G$ with compact supports and $\psi$ a fixed non-zero function in $L^2(G),$ usually called window function. The continuous Gabor transform of $ f $ with respect to $\psi$ can be defined as a measurable field of operators on $G \times \widehat{G}$ given by
 \begin{equation}\label{eq1}
  G_{\psi}f(x,\pi) = \int_{G} f(y)\overline{\psi(x^{-1}y)}\pi(y)^* dm(y)
 \end{equation}
    
 The operator valued integral \eqref{eq1} is considered in the weak sense, i.e., for each $(x,\pi) \in G \times \widehat{G}$ and $\xi,\eta \in H_{\pi},$ we have that
 \[ \langle G_{\psi}f(x,\pi)\xi,\eta \rangle = \int_G f(y)\overline{\psi(x^{-1}y)}\langle \pi(x)^*\xi,\eta \rangle dm(x)d\mu(\pi).
 \]
 For each $x \in G,$ let $f_{x}^{\psi}: G \to \mathbb{C}$ by 
 \begin{align}
 f_{x}^{\psi}(y) =  f(y)\overline{\psi(x^{-1}y)}.
 \end{align}
 Since $f \in C_c(G)$ and $\psi \in L^2(G),\text{ Therefore } f_x^{\psi}\in L^1(G)\cap
L^2(G)  \text{ for all }  x\in G$. The Fourier transform of $f^{\psi}_x$ is given by 
\[ \widehat{f_x^{\psi}}(\pi) = \int_G f_x^{\psi}(y)\pi(y)^* dy = \int_G f(y)\overline{\psi(x^{-1}y)}\pi(y)^* dy = G_{\psi}f(x,\pi). 
\] 
Also, using the Plancherel Formula \cite{folland}, it follows that $\widehat{f_x^{\psi}}$is a Hilbert Schmidt operator for almost  all $\pi \in \widehat{G}.$ Therefore, $G_{\psi}f(x,\pi)$ is a Hilbert Schmidt operator for all $x \in G$. Also for almost all $\pi \in \widehat{G}\text{  and } f,\psi \in L^2(G)$ \cite{ashish1}, we have 
\[ ||G_{\psi}f||_{H^2(G\times \widehat{G})} = ||\psi||_2||f||_2. \]

\noindent We investigate several classes of locally compact groups for the following  so called Qualitative Uncertainty Principle (QUP) for Gabor transform :

 \noindent If  $f \in L^2(G)\text{ and }\psi$  is a window function satisfying 
 $$ m\times \mu{\{ (x,\pi):G_{\psi}f(x,\pi)\neq 0\}< \infty, \text{ then } f = 0 \text{ a.e.}} $$ 
 
\noindent In \cite{ashish}, QUP was shown to hold for various class of locally compact groups like abelian groups, compact extensions of groups having QUP.

\noindent  In section 2 we shall prove QUP for  Gabor transform  holds for groups of the form $ H \times D, D$ being a discrete group and $H$ is a group for which QUP holds.
 In section 3 we show that QUP  hold for certain  classes of nilpotent Lie groups including Heisenberg group and other low dimensional nilpotent Lie group. In section 4, a weaker version of QUP has been  discussed for Moore groups. In last section we give necessary condition for  QUP  to hold for Gabor transform for multiplier extension of $\mathbb{T}$ by abelian group.  
    
\section{QUP for Gabor Transform}
\noindent In this section, $G$ will be second countable, unimodular, locally compact group of type I. For $\pi \in \widehat{G}, \mathcal{H_{\pi}} $ will denote the Hilbert space of $\pi.$

  \begin{thm}\label{thm1}
  Let $H$ be a subgroup of finite index in $G$ then  $H$ has QUP for Gabor transform if and only if  $G$ has QUP for Gabor transform.
  \end{thm}
  
  \begin{proof}
  
 Let QUP for Gabor transform hold for $G$.  $H$ contains a subgroup $N$ of finite index which is normal in $G$. Since $ H/N $ is finite, so by \cite{ashish} we can assume $H$ to be normal. Let $ f \in L^2(H)$ and $ \psi $ be a window function such that, 
 \[ m\times \mu \{(x,\pi)\in H\times\widehat{H}: G_\psi f(x,\pi) \neq 0\} < \infty. 
\]
So, there exist a zero set $K \text{ in } H$ such that for all $ x \in H\setminus K ,$
\[ \mu {\{ \pi : G_{\psi}f(x,\pi)\neq 0 \}} < \infty. \]
Define $F(x) = f(x)\chi_{H}(x) \text{ and } \Psi(x)= \psi(x)\chi_{H}(x),\text { for all } x \in G.  $ Then
\begin{align*}
G_\Psi F(x,\pi)= \int_H f(y)\overline{\psi(x^{-1}y)}\chi_{H}(x^{-1}y)\pi(y)^*dy. 
\end{align*}
Clearly if  $x  \notin H, \text{ then }  G_\Psi F(x,\pi)=0.$  Now for  $ x\in H, \text{ define } F_x^\Psi(y) = f_{x}^{\psi}(y)\chi_{H}(y),  \text{for all } y \in G $. Fix $x \in  H\setminus K $ and  using the Plancherel formula \cite{klep},  it follows
 \begin{equation}\label{eq2}
 \int_{\widehat{G}} \chi_{\{\pi \in \widehat{G}: G_\Psi F(x,\pi) \neq 0 \}}(\pi)d\pi \leq [G:H]^3 \int_{\widehat{H}}\chi_{\{ \pi \in \widehat{H} : G_\psi f(x,\pi)\neq 0 \}}(\pi)d\pi.
\end{equation}  

 \noindent Equation \eqref{eq2} holds for  every  $x \in H\setminus K.$ Therefore, we have 
     \begin{align*}
       &\int_H\int_{\widehat{G}} \chi_{\{\pi \in \widehat{G}: G_\Psi F(x,\pi) \neq 0 \}}(x,\pi) dxd\pi \\&\leq [G:H]^3\int_H\int_{\widehat{H}}\chi_{\{ \pi \in \widehat{H} : G_\psi f(x,\pi)\neq 0 \}}(x,\pi)dxd\pi < \infty.
     \end{align*}
    \noindent So, it follows \[  \int_H\int_{\widehat{G}}\chi_{\{(x,\pi)\in H\times\widehat{G}: G_\Psi F(x,\pi)\neq0\}}(x,\pi)dxd\pi  < \infty,\]
  and therefore we have \[   \int_G\int_{\widehat{G}}\chi_{\{(x,\pi)\in G\times\widehat{G}: G_\Psi F(x,\pi)\neq0\}}(x,\pi)dxd\pi \\
  < \infty. \]  

\noindent Since $ G$  has QUP for Gabor transform Therefore  $F=0$ a.e. and consequently   $ f = 0$ a.e.

\noindent Conversely  let $H$ be a subgroup of finite index and has QUP for Gabor transform. 
 Also $H$ has a subgroup $N$ of finite index normal in $G$. Since $H$ has QUP for Gabor transform and $N$ has finite index in $H$, therefore $N$ has QUP for Gabor transform.
As $G/N$  is  compact, So by \cite{ashish}, $G$ has QUP for Gabor transform.
  \end{proof}
  
\begin{rem}
If $G$ is a group having all irreducible representation bounded i.e. there exist $M > 0$ such that dim $\pi \in \widehat{G} $ then by \cite{moore}, $G$ has an abelian subgroup H of finite. Thus it follows from above theorem that $G$ has QUP for Gabor transform if and only if $G_{0}$ is non compact.  
\end{rem}

\noindent Next, we consider groups of the form  $ G = H\times D,  $ where $H$  be a second countable, unimodular, locally compact group of type I and $D$ be a discrete type I group. We now consider Wiener amalgam space \cite{four}, which is defined as follows:
\[
 W(L^2,l^1)(G)= \{ f:G \to \mathbb{C}| f \text{ is measurable, } \sum\limits_{t \in D}\left(\int\limits_H| f(x,t)| ^2dx\right)^{\frac{1}{2}} < \infty \}.
\]
 Note that $W(L^2,l^1)(G) = L^2(G)$ if $D$ is finite.
\begin{lem} \label{pro5}
$W(L^2,l^1)(G)\subseteq L^2(G)$. However the containment may be strict.
\end{lem}

\begin{proof}
For $ f \in W(L^2,l^1)(G)$ \text{ and } $t \in D,$ consider $b_t = \left(\int\limits_H | f(x,t)|^2 dx \right)^{\frac{1}{2}}$.
 Now  $ 0\leq b_t < \infty  $ for all $t \in D.\text{ Let } S = \{t: b_t\geq 1\}
 \text{ and } |S| = \sum\limits_{t\in S}1 \leq \sum\limits_{t\in S }b_t \leq \sum\limits_{t \in D}b_t <\infty $ which imply $b_t\geq 1$  for only finitely many $t$.
  Also $ b_t > 0 $ for at most countably many $t$. So we have $\sum\limits_{t\in D}\int\limits_H | f(x,t)|^2 dx  = \sum\limits_{t \in D}b_t^2 = \sum\limits_{t \in S}b_t^2 + \sum\limits_{t\in S'}b_t^2\leq \sum\limits_{t \in S}b_t^2 + \sum\limits_{t\in S'}b_t <\sum\limits_{t \in S}b_t^2 + \sum\limits_{t\in D}b_t <\infty $ which implies that $f \in L^2(G).$ 
\end{proof}

\noindent For the strict inclusion, consider $G =  \mathbb{R} \times \mathbb{Z} $ and 
\[ 
   f(x,n)=
       \begin{cases}
       \frac{1}{n}, & x\in [0,1]\text{ and } n\geq 1\\
       =0, & \text{ elsewhere }.
       \end{cases}
\]  
Clearly $f \in L^2(G) \text{ but } f \notin W(L^2,l^1)(G).$

\noindent We will say $ G \text{ has }  (QUP )'$ for Gabor transform if  $f,  \psi(\neq 0)  \in W(L^2,l^1)(G)$ such that 
 \[
  m_{G}\times \mu_{G}\{ (x,\pi): G_{\psi}f(x,\pi)\neq 0 \} < \infty 
  \]
   then $f = 0 $ a.e. 
 
  \noindent It may be noted that theorem \ref{thm1} remain valid for   $(QUP )'.$
\begin{thm}\label{th1}
If $H$ has QUP for Gabor transform then $G$  has $(\text{QUP})'$ for Gabor transform.
\end{thm}
 
 \begin{proof}
 
 First of all, we assume  $D$ is  abelian. 
 Let $f, \psi(\neq 0 )  \in W(L^2,l^1)(G)$ be such that  
 \[m_{G}\times \mu_{G}\{ (x,\pi): G_{\psi}f(x,\pi)\neq 0 \} < \infty. \] 
 By Minkowski  inequality, we obtain
 \[ \int_H\left(\sum\limits_{t \in D}| \psi(x,t)|\right)^2dx \leq \left( \sum\limits_{t \in D}\left(\int\limits_H| \psi(x)| ^2dx\right)^{\frac{1}{2}}\right)^2 < \infty.
 \]
 So, for almost every $x \in H, \sum\limits_{t \in D}| \psi(x,t)| < \infty.
\text{ For } \gamma \in \widehat{D}, \text{ define } \tilde{\psi}(y) = \sum\limits_{t \in D} \psi(y,t)\gamma(t), \text {for all } y \in H.$ Again by Minkowski  inequality, it follows that $\tilde{\psi} \in L^2(G)$. Since $\psi \neq 0$, therefore there exist $\gamma \in \widehat{D}$ such that $\tilde{\psi} \neq 0.$ For such  $\gamma \in \widehat{D}$, define $\phi (x,t) = \psi(x,t)\gamma(t), \text{ for } (x,t)\in H \times D$. Then $ \phi \in W(L^2,l^1)(G)$ and 
\begin{align*}
G_{\phi}f(x,t,\pi,\delta) &= \int_H\sum\limits_{s\in D} f(y,s)\overline{\phi(x^{-1}y,t^{-1}s)}\pi(y)^* \overline{\delta(s)}dy\\
                 &= \int_H\sum\limits_{s\in D} f(y,s)\overline{\psi(x^{-1}y,t^{-1}s)}\ \overline{\gamma(t^{-1}s)}\pi(y)^* \overline{\delta(s)}dy\\
                 &=\gamma(t)G_{\psi}f(x,t,\pi,\gamma\delta).
\end{align*} 
Thus, we have  $$\{(x,t,\pi,\delta):G_{\phi}f(x,t,\pi,\delta)\neq 0 \} = \{(x,t,\pi,\overline{\gamma}\delta): G_{\psi}f(x,t,\pi,\delta) \neq 0 \}.$$ which implies that 
\[m_{G}\times m_{\widehat{G}}\{(x,t,\pi,\delta):G_{\phi}f(x,t,\pi,\delta)\neq 0 \} < \infty. \]
 Therefore there exist a zero set $K \text{ in } \widehat{D} \text{ such that for all } \delta \in \widehat{D}\setminus K,$
 \begin{align}
 \iint\limits_{H \widehat{H}}\sum\limits_{s\in D} \chi_{\{(x,t,\pi):G_{\phi}f(x,t,\pi,\delta)\neq 0 \}}dxd\pi < \infty.
\end{align} 
 Now for each $\delta \in \widehat{D}\setminus K $, define $\tilde{f}(x)= \sum_{t\in D}f(x,t)\delta(t)$, for all $ x \in$ H. Then $\tilde{f}\in L^2(H)$ also
 \begin{align*}
  G_{\tilde{\psi}}\tilde{f}(x,\pi) &= \int\limits_{H } \tilde{f}(y)\overline{\tilde{\psi}(x^{-1}y)}\pi(y)^* dy\\
                    &=\int\limits_{H}\sum\limits_{s \in D }f(x,s)\delta(s)\sum\limits_{t \in D}\overline{\psi(x^{-1}y,t)}\ \overline{\gamma(t)}\pi(y)^*dy\\
                    &=\sum\limits_{t \in D}\int\limits_{H}\sum\limits_{s \in D }f(x,s)\delta(s)\overline{\psi(x^{-1}y,t^{-1}s)}\ \overline{\gamma(t^{-1}s)}\pi(y)^*dy\\
                    &=\sum\limits_{t \in D}G_{\phi}f(x,t,\pi,\delta).
  \end{align*}
  If $G_{\tilde{\psi}}\tilde{f}(x,\pi) \neq 0 \text{ then } G_{\phi}f(x,t,\pi,\delta) \neq 0 \text{ for all } t \in M \subset D, \text{ where } |M|\geq 1$ and 
\begin{align*}
m_H\times \mu_{H}{\{(x,\pi):G_{\tilde{\psi}}\tilde{f}(x,\pi) \neq 0\}} &= \iint\limits_{H \widehat{H}}\chi_{\{(x,\pi):G_{\tilde{\psi}}\tilde{f}(x,\pi)\neq 0 \}}(x,\pi)dxd\pi\\
   &\leq \iint\limits_{H \widehat{H}}\sum\limits_{s\in D} \chi_{\{(x,t,\pi):G_{\phi}f(x,t,\pi,\delta)\neq 0 \}}dxd\pi < \infty.
\end{align*}
Since $H$ has QUP for Gabor Transform, it follows that $\tilde{f}(x)= \sum\limits_{t\in D}f(x,t)\delta(t)= 0 \text{ a.e. for all } \delta \in \widehat{D}\setminus K$ which implies that  $f=0$ a.e.

\noindent Now if $D$ is discrete type-I, then $D$ has an abelian subgroup $A$ of finite index. So, $H \times A $ has $(QUP)'$ for Gabor transform and $H \times A$  has finite index in $H\times D$. So, using Theorem \ref{thm1} $H \times D$ has $(QUP)'.$ 
\end{proof}

\begin{rem}
  
 \noindent $(i)$ If $G =  H \times D $ where $D$ is discrete group such that $G$ has QUP for Gabor transform then so does $H$. Consider $f \in L^2(H),\text{ and } \psi$ be a window function such that
 \begin{align*}
  m_{H}\times\mu_{H}{\{(x,\pi):G_{\psi}f(x,\pi)\neq 0\}} <\infty. 
 \end{align*}
 Define $g(x,t) = f(x)\chi_{\{e\}}(t) \text{ and } \phi(x,t) = \psi(x)\chi_{\{e\}}(t)$.
 Then $g,\phi \in L^2(G)$ and  
 \[
   G_{\phi}g(x,t,\pi,\delta)=
           \begin{cases}
                G_{\psi}f(x,\pi),  & if\  t = e\\
                 0, & elsewhere
            \end{cases}  
\]
  Therefore,
  \begin{align*}
  & m_{G}\times \mu_{G} \{(x,t,\pi,\delta):G_{\phi}g(x,\pi)\neq 0 \}\\& = m_{H}\times \mu_{H}\{(x,\pi):G_{\psi}f(x,\pi)\neq 0\}m_{D}(\widehat{D})< \infty
\end{align*}  
which implies that $g = 0$ a.e. and hence  $f =0 $ a.e.

\noindent $(ii)$ If $G = H \times D, \text{ where } D$ is a finite group then by Theorem \ref{th1} one can see if $ H$ has QUP for Gabor transform then so does $G.$
\end{rem}
 
\begin{thm}\label{th2}
 Let $G$ be a noncompact, nondiscrete, unimodular type-I group.
 If $G$ has a compact open subgroup $H$, Then  QUP for Gabor transform does not hold for $G.$
 \end{thm}
 
 \begin{proof}
Let $\alpha = m_G(H), \text{ then } 0 <\alpha <\infty \text{ and } m_H = \alpha^{-1}(m_G|_H) $ is a Haar measure for $H$. Also $H$ is non-discrete. So, $H$ is non-trivial and $A(\widehat{G},H)\subset\widehat{ G}.$
 Define $ f =  \chi_H \text{ and } \psi = \chi_H $.
Then $f$ and \ $\psi $ are  non-zero function of $L^2(G).$   
If $ \pi \in A(G,\widehat{H})$ then $ G_\psi f(x,\pi) = \chi_H(x)\alpha I.$ 
    
\noindent Let $ \pi \notin A(\widehat{G},H),  a \in H \text{ and }  \zeta,\eta \in \mathcal{H}_{\pi}$ it follows, 
  \begin{align*}
 \langle G_\psi f(x,\pi)\zeta, \eta \rangle =  \langle G_\psi f(x,\pi)\pi(a)^*\zeta, \eta \rangle
  \end{align*}                
    which means, $  G_\psi f(x,\pi)(I - \pi(a)^*)\zeta = 0  \text{for all  } \zeta \in  \mathcal{H}_\pi,a \in H.$ Therefore $G_\psi f(x,\pi)\xi = 0  \text{ for all }   \xi \in  V$ where $V$ is the smallest closed set containing $ \bigcup\limits_{a\in H}(I-\pi(a)^*)(\mathcal{H}_{\pi}), a\in H. $  
         
 \noindent Now for all $ g \in G,a \in H, \xi \in \mathcal{ H}_\pi$, We have
  \begin{align*}
  \pi(g)(I - \pi(a)^*)\xi   &= (I - \pi(ga^{-1}g^{-1})^*)\pi(g)\xi
  \end{align*}
      which implies that  $V$ is a closed invariant subspace, as $\pi$ is  irreducible so  $V = \mathcal{H}_\pi.$Thus it follows
     \[
   {\{(x,\pi):G_\psi f(x,\pi)\neq 0 \}} = \chi_{H}\times A(\widehat{G},H).\]
  Consequently, using (\cite{hogan}, Theorem 2.4)we have 
      \[ m\times m_{\widehat{G}} \{(x,\pi): G_\psi f(x.\pi) \neq 0 \} = m(H)m_{\widehat{G}}(A(\widehat{G},H)) < \infty.   \]
 \end{proof}
 
 \begin{rem}
   A locally compact group is said to be Plancherel if dual object $\widehat{G}$ can be equipped with a measure $ \mu_{G} \ ($the Plancherel measure$)$ such that
            \[ \int_G |f(x)|^2dm_G(x) = \int _{\widehat{G}}\| \widehat{f}(\pi)^2\|d\mu_{G}
(\pi). \] Theorem \ref{th2} is true for noncompact, nondiscrete Plancherel group. Maurtner Groups {\cite{folland}} are example of Plancherel group, unimodular  not type I.
\end{rem}
 \begin{rem}
Clearly,  QUP   for Gabor transform  does not hold for groups of the type $ D \times  K \text{ where } D $ is discrete Maurtner group, $ K $ is compact group and Moore group with compact component of identity.
 \end{rem}

\section{ Nilpotent Group }
 For a locally compact unimodular group  G of type I having  center $Z$,  suppose that there exist a zero set $E$ in $\widehat{Z}$ such that for every $\lambda \in \widehat{Z}\setminus E,$ the induced representation ind$_{Z}^{G} \lambda $ is a multiple of an irreducible $\pi_{\lambda}$. Then according  the Plancherel formula 
(\cite{klep}, Theorem 8.1), the Plancherel measure on $\widehat{G}$ is given by 
 \[\mu_{G}(W)= \int\limits_{\widehat{Z}\setminus E}\chi_{W}(\pi_{\lambda})q(\lambda)d\lambda,
 \]
 where $ W \subseteq  \widehat{G} \text{ and } q$ is a positive measurable function.
  \begin{prop}\label{prop1}
 In the above situation, suppose that $Z$ has the following property: 
 
 If $f \in L^{2}(Z)\text{ and } \psi$ is a window function in $L^2(Z)$ satisfies
 
\[\int\limits_{Z } \int\limits_{\widehat{Z}\setminus E }\chi_{\{(x, \lambda ): G_{\psi}f(x,\lambda)\neq 0 \}}(x,\lambda)q(\lambda) d\lambda dx < \infty, \] implies $f = 0$ a.e. Then QUP for Gabor transform hold for $G$.
  \end{prop}
  
  \begin{proof}
  Let $g \in L^2(G) \text{ and } \phi$ be a window function such that, 
  \[m\times \mu\{(x,\pi):G_{\phi}g(x,\pi)\neq 0\} < \infty.
  \]
  By Weil's formula, it follows
  \[ \int\limits_{G/Z}\int\limits_{ Z }\int\limits_{ \widehat{G}} \chi_{\{(hx,\pi): G_{\phi}g(x,\pi)\neq 0 \}} (hx,\pi)d\pi dh d\dot{x} < \infty.\]
  Now there exist a zero set $K$ in $G$ such that for all $ x\in G\setminus K,$ 
  \[\iint\limits_{Z \widehat{G}} \chi_{\{(hx,\pi): G_{\phi}g(x,\pi)\neq 0 \}} (hx,\pi)d\pi dh  < \infty.\]
  Also there exist a  zero set $M$ such that for all $y\in G\setminus M , \text{ we have } (\phi_y|_Z)\text{ and } (g_y|_Z) \in L^2(Z)$ where $(g_y|_Z)(h) = g(hy), \text{for all }  h \in Z.$ Fix $ y \in G\setminus M $ and define $ k(x) = \phi(x^{-1}y),$ for all $x \in $ G. Then $k \in L^2(G).$ 
  Again by  Weil's formula, there exist a  zero set $N_y$ such that for all $x \text{ in } G\setminus N_y$,
  \[\int_{Z}| k(xv)|^2 dv =  \int_{Z} | \phi(x^{-1}vy)|^2 dv < \infty. \]
Since $ \phi \neq 0 $, Therefore  we can choose $x \in G\setminus (K\cup N_y)$ such that $0 \neq ({ }_{x}\phi)_y|_Z \in L^2(Z),\text{ where } ({ }_x\phi)_y|_Z(v)= \phi(x^{-1}vy).$
Fix  such $x \in G\setminus (K\cup N_y)$ and for $h \in Z ,\text{ define } g_{hx}^{\phi}(z) = g(z)\overline{\phi(x^{-1}h^{-1}z)}.$ Also $(g_{hx}^{\phi})_y|_Z = (g_y|_Z)_{h}^{({ }_x\phi)_y|_Z}. $ 

Now using the Plancherel formula we have,
\[ \int_{\widehat{Z}\setminus E}\chi_{\{\lambda :G_{({ }_x\phi)_y|_Z)}(g_y|_Z)(h,\lambda) \neq 0\}}(\lambda)q(\lambda) d\lambda \leq \int_{\widehat{G}} \chi_{\{\pi: G_{\phi}g(hx,\pi) \neq 0\}}(\pi)d\pi.
\] 
 On integration it follows,
\[\iint\limits_{ Z \widehat{Z}\setminus E}\chi_{\{(h,\lambda) :G_{({ }_x\phi)_y|_Z)}(g_y|_Z)(h,\lambda) \neq 0\}}(h,\lambda)q(\lambda) d\lambda dh \leq \iint\limits_{Z \widehat{G}} \chi_{\{(h,\pi): G_{\phi}g(hx,\pi) \neq 0\}}(h,\pi)d\pi dh < \infty .
\] which implies that $g_{y}|_Z = 0 $ a.e. for all $y \in G\setminus M$. Since $y$ was arbitrary fixed,
Therefore,  $ g = 0 $ a.e.
 \end{proof}

For $ f \in L^1(G), $ let us define 
\[ A_f = \{ x \in G : f(x) \neq 0 \} \text { and } B_{f} = \{ \pi \in \widehat{G} : \widehat{f}(\pi) \neq 0 \}.\]
\noindent  The idea of the proof of following lemma emerges from \cite{ajay} and \cite{bened}.

\begin{lem}\label{lem1}
 Let $ f \in L^1(\mathbb{R}^n), n\geq 2  \text{ be such that for some  } \delta \in \mathbb{R}^n, $
\[\int\limits_{\mathbb{R}^n} \chi_{A_f}(x)|p(x-\delta)|dx < \infty \text{ and } \int\limits_{\mathbb{R}^n} \chi_{B_f}(\xi)|p(\xi-\delta)|d\xi <  \infty,  
\]
where $ p(x) = p(x_1, \dots x_n) = x_1, {x_1}^2, x_1 x_2, {x_1}^2+ {x_2}^2.$ Then $f = 0 $ a.e.
\end{lem}

\begin{proof}
  Replacing $f $ by a suitable dilate we can assume that
    $ $

   \[ \int\limits_{\mathbb{R}^n} \chi_{A_{f}}(x)| p(x - \delta)| dx < \tilde{\mu}_{\mathbb{T}^2} \times m_{\mathbb{T}^{n-2}}(\mathbb{T}^n)\] 
   where $\mathbb{T}^{n}  = \mathbb{R}^{n}/\mathbb{Z}^{n} \text{ and } d\tilde{\mu}_{\mathbb{T}^2}(y) = p_{\delta}(y)dm_{\mathbb{T}^2}(y), $

\noindent $p_{\delta}(y) =\text{min}\{y_1, (1-y_1),(y_1 + \lceil {\delta_1}\rceil-\delta_1)\},  \text{ when }p(x) = x_1,\\
p_{\delta}(y) =\text{min} \{y_1^2, (1-y_1)^2,(y_1 + \lceil {\delta_1}\rceil-\delta_1)^2\}, \text{ when }  p(x) = {x_1}^2, $
\begin{align*}
p_{\delta}(y) =\text{min} \{ &y_1y_2, (1-y_1)y_2,y_1(1-y_2),(1-y_1)(1-y_2),y_1|y_2 + \lceil {\delta_2}\rceil                      -\delta_2|,\\
& (1-y_1)|y_2 + \lceil {\delta_2}\rceil -\delta_2|,y_1 + \lceil {\delta_1}\rceil -\delta_1|y_2,|y_1 + \lceil {\delta_1}\rceil -\delta_1||y_2 + \lceil {\delta_2}\rceil -\delta_2|,\\
&|y_1 + \lceil {\delta_1}\rceil -\delta_1|(1-y_2)\},  \text{ when } p(x) = x_1x_2,
\end{align*}
\begin{align*}
p_{\delta}(y) =\text{min}\{&{y_1}^2+{y_2}^2, (1-y_1)^2 + {y_2}^2,{y_1}^2+(1-y_2)^2,(1-y_1)^2+(1-y_2)^2,\\
&{y_1}^2+(y_2 + \lceil {\delta_2}\rceil -\delta_2)^2, (1-y_1)^2+(y_2 + \lceil {\delta_2}\rceil -\delta_2)^2,(y_1 + \lceil {\delta_1}\rceil -\delta_1)^2{y_2}^2,\\
&(y_1 + \lceil {\delta_1}\rceil -\delta_1)^2+(y_2 + \lceil {\delta_2}\rceil -\delta_2)^2,(y_1 + \lceil {\delta_1}\rceil -y_1)^2+(1-y_2)^2\}, \\
& \text{ when } p(x) = {x_1}^2+{x_2}^2
\end{align*} 
 $\lceil . \rceil $  being the greatest integer function. 
    Now define $\phi: \mathbb{R}^n \rightarrow \mathbb{R} \cup \{ \infty \}$ by
   
$$  \phi(\xi) = \sum\limits_{m \in \mathbb{Z}^n}\chi_{B_f}(\xi + m)|p( \xi+ m)|,$$
 and let $ K = \{\xi \in \mathbb{R}^n : 0 < \xi_i < 1, i = 1,2 \dots n\}.$
 
 \noindent Then $\int\limits_{K}\phi (\xi)d\xi = \int\limits_{\mathbb{R}^n}\chi_{B_{f}}(\xi)|p( \xi)| d\xi <\infty.$ Therefore, $\phi(\xi) < \infty $ a.e.
   
   \noindent Now for almost all   $ \xi  = (\xi_1, \dots \xi_n) \in K$ and
    $m \in \mathbb{Z}^n$,
   \[ |p( \xi - \delta + m)| \geq  p_{\delta}(\xi_1, \xi_2) > 0 \]
and   $  p_{\delta}(\xi_1, \xi_2) \sum\limits_{ m \in \mathbb{Z}^n}\chi_{B_{f}}(\xi+m) \leq \phi(\xi) < \infty.$

\noindent Consequently, for almost all $\xi \in \mathbb{R}^n, \chi_{B_{f}}(\xi + m)\neq 0 $ only for finitely many $m \in \mathbb{Z}^n.$  
Fix $\xi \in \mathbb{R}^n$ and define $\tilde{f}_{\xi}(x) = \sum\limits_{m \in Z^n} e^{-\iota \langle \xi, x - m\rangle}f(x-m).$ Then $\tilde{f}_{\xi} \in L^1(\mathbb{T}^n). $ Also,
$\widehat{(\tilde{f}_{\xi})}(k) = \widehat{f}(\xi+k).$

 \noindent Now, 
\begin{align*}
 \tilde{\mu}_{\mathbb{T}^2}\times m_{\mathbb{T}^{n-2}}(\mathbb{T}^n) &> \int\limits_{\mathbb{R}^n}\chi_{A_f}(x)| p(x - \delta) | dx\\
                 &= \int\limits_{\mathbb{T}^n}\sum\limits_{m \in \mathbb{Z}^n}\chi_{A_f}(x + m)|p( x - \delta + m)| d\dot{x}\\
                 &\geq \int\limits_{\mathbb{T}^n}\chi_{A_{\tilde{f}_{\xi}}}(\dot{x_1},\dots ,\dot{x_n})| p_{\delta}(\dot{x}_1,\dot{x}_2)| d\dot{x_1}\dots  d\dot{x_n}\\
                  &= \tilde{\mu}_{\mathbb{T}^2}\times m_{\mathbb{T}^{n-2}} (A_{f_{\xi}})
\end{align*}   
        Thus  $ \tilde{\mu}_{\mathbb{T}^2}\times m_{\mathbb{T}^{n-2}}\{x : \tilde{f_{\xi}}(x) = 0 \}>  0  $  which implies that $m_{\mathbb{T}^n}\{x : \tilde{f_{\xi}} = 0 \}> 0.$ But for almost all $\xi \in \mathbb{R}, \tilde{f}_{\xi}$ is a trigonometric polynomial. Therefore  $\widehat{f} = 0 $ a.e. So, by  Fourier inversion theorem,  $f = 0 $ a.e.
 \end{proof}

\begin{lem}\label{lem2}
Let $f \in L^2(\mathbb{R}^n )\text{ and } \psi $ be a window function such that 
\[\int\limits_{\mathbb{R}^{2n}}\chi_{\{(x,\xi): G_{\psi}f(x,\xi)\neq 0\}}(x, \xi)| p(\xi) | dxd\xi < \infty ,\] where $p(\xi) = p(\xi_1, \dots \xi_n ) = \xi_1,\xi^2, \xi_1\xi_2, {\xi_1}^2 + {\xi_2}^2$,
 then $f = 0$ a.e.
 \end{lem}
 \begin{proof}
 For each $(y,\delta) \in \mathbb{R}^n
 \times \mathbb{R}^n,$ define 
$$F_{(y,\delta)}(x,\xi)=  \xi(x)G_{\psi}(M_{\delta}T_yf)(x,\xi) \times G_{\psi}(M_{\delta}T_yf)(-x,-\xi),$$ where $ T_y, M_\delta $ are translation and modulation operator respectively. 
Then it is easy to prove that  $\widehat{F_{(y,\delta)}}(\xi,z)= F_{(y,\delta)}(-z,\xi).$

\noindent Moreover $F_{(y,\delta)}$ is continuous and 
\begin{align*}
\{(x,\xi): F_{(y,\delta)}(x,\xi)\neq 0 \} &\subseteq \{(x,\xi): G_{\psi}(M_{\delta}T_y f)(x,\xi)\neq 0 \} \\
&= \{(x,\xi): G_{\psi} f(x-y,\xi - \delta)\neq 0 \}.\\
\end{align*}
So it follows,
$\chi_{A_F}(x,\xi) \leq \chi_{\{(x,\xi): G_{\psi}(M_{\delta}T_y f)(x,\xi)\neq 0 \}}(x-y,\xi-\delta)$
and hence,
\begin{align*}
\int\limits_{\mathbb{R}^{2n}}\chi_{A_F}(x,\xi)|p( \xi - \delta )| dxd\xi
        & \leq \int\limits_{\mathbb{R}^{2n}}\chi_{\{(x,\xi): G_{\psi}f(x,\xi)\neq 0 \}}(x-y,\xi-\delta) | p(\xi - \delta) | dxd\xi\\
        &=  \int\limits_{\mathbb{R}^{2n}}\chi_{\{(x,\xi): G_{\psi}f(x,\xi)\neq 0 \}}(x,\xi) |p( \xi )  | dxd\xi < \infty.
\end{align*}
 Moreover, $\int\limits_{\mathbb{R}^{2n}}\chi_{A_F}(x,\xi)|p( \xi - \delta) | dxd\xi = \int\limits_{\mathbb{R}^{2n}}\chi_{B_F}(\xi,x)|p( \xi - \delta ) | dxd\xi< \infty.$\linebreak
  Now using lemma \ref{lem1} \ $F_{(y,\delta)} = 0 $ for every $(y,\delta) \in \mathbb{R}^2.$\\ 
  So, $F_{(y,\delta)}(0,0) =  (G_{\psi} f(-y,-\delta))^2 = 0$  for every $(y,\delta) \in \mathbb{R}^2$ which implies that $f = 0$ a.e.
 \end{proof}

For a simply connected nilpotent Lie group G with Lie algebra $\mathcal{G}, \widehat{G}$ can be  parametreized by the set of coadjoint orbits of $G$ in the vector space dual $\mathcal{G}^* \text{ of } \mathcal{G}$\cite{kiri}. Let $f \in \mathcal{G}^*$, and  $\pi_{f}$  denote the irreducible representation associated to $f$ and by $O_f$ the coadjoint orbit of $f$. Moreover, let Z and $\mathcal{Z}$ denote the center of $G \text{ and } \mathcal{G},$respectively. 
Applying  Proposition \ref{prop1} to $G$, we obtain QUP  for G if,

\noindent (i) For allmost all $f \in \mathcal{G}^*, \pi_{f}  \sim \text{ind}_{Z}^{G}(\pi_{f}|_Z). $

\noindent (ii) dim $\mathcal{Z} \leq 2, q(\lambda) \text{ of Proposition } \ref{prop1} \text{ is of the form  } q(\lambda) = |p(\lambda)|, $ where $p$ is homogeneous polynomial of degree $\leq 2.$

\noindent Now, from the data presented in \cite{niel} for low dimensional groups  and using lemma   lemma \ref{lem2}the following groups have QUP for Gabor transform:
\noindent $G_3, G_{5,1}, G_{5,3}, G_{5,6}, $ for $n = 1$ and 
 $G_{6,16},G_{6,17},G_{6,19},G_{6,20},G_{6,21},G_{6,22},G_{6,23},G_{6,24}$  for $ n =2$  have QUP for Gabor transform.
\section{Moore Group }
\noindent  A locally compact group $G$ is said to be Moore if dim $\pi < \infty $ for all $\pi \in \widehat{G}.$ This class contain finite extension of abelian groups and compact groups. Let $G_{F}$ denote  the subgroup consisting of all elements of G with relatively compact conjugacy classes. It is an open normal subgroup of finite index. 

 \noindent A locally compact $G$ has weak QUP for Gabor Transform if  $f \in L^2(G),\psi$  be a window function satisfying 
 \[ m\times \mu{\{ (x,\pi):G_{\psi}f(x,\pi)\neq 0\}< 1, \text{ then } f = 0 \text{ a.e.}} \]
 
 \begin{prop}\label{pro1}
  Let $G$ be a locally compact Plancherel Group and $K$ is a compact normal subgroup of $G$. If $G$ satisfies QUP $($the weak QUP $)$ of Gabor transform then $G/K$ also satisfy QUP $( \text{weak QUP})$ of Gabor transform.
 \end{prop}
 
 \begin{proof}
 Let $ \dot{f} \in L^2(G/K)$\text{ and } $ \dot{\psi} $ be a window function 
  such that
  \[
m_{G/K}\times \mu_{G/K}\{(\dot{x}, \dot{\pi}):G_{\dot{\psi}} \dot{f}(\dot{x},\dot{\pi})\neq 0 \} < \infty (<1).
\]
Define $f(x)= \dot{f}(q(x))\text{  and }  \psi(x)=\dot{\psi}(q(x)),$  where $q$ is the quotient map.
 
\noindent  For $\xi, \eta \in \mathcal{H}_{\pi}$ consider,
 \begin{align*}
   \langle G_\psi f(x,\pi)\xi,\eta \rangle   &=\iint\limits_{G/K K}f(yk)\overline{\psi(x^{-1}yk)}\langle\pi(k)^*\pi(y)^*\xi,\eta \rangle dk d\dot{y} \\
&=\int_{G/K} \dot{f}(q(y))\overline{\dot{\psi}(q(x^{-1}y))} \left(\int_K \langle \pi(k)^*\pi(y)^*\xi,\eta \rangle dk\right) d\dot{y}
 \end{align*}
 So, it follows that,
 \[
 \langle G_\psi f(x,\pi) \xi,\eta \rangle = 
 \begin{cases}
 \langle G_{\dot{\psi}}\dot{f}(\dot{x},\pi)\xi,\eta\rangle, & if\  \pi \in A(\widehat{G},K) \\
 0, & \text{ elsewhere}
 \end{cases}
 \]
 which means that  $G_\psi f(x,\pi)$ is either zero a.e. on cosets of $K$ or non zero a.e.  Now consider,
\begin{align*}
 &m_{G/K}\times \mu_{G/K} { \{ (\dot{x},\pi): G_{\dot{\psi}} \dot{f}(\dot{x},\pi) \neq 0 \}}\\
                 & = \iint\limits_{\widehat{G/K} G/K}\chi_{\{(\dot{x},\pi):G_{\dot{\psi}}\dot{f}(\dot{x},\pi)\neq 0\}}(\dot{x},\pi)d\dot{x}d\pi\\
                 & = \iiint\limits_{\widehat{G/K} G/K K}\chi_{\{(xk,\pi):G_{\dot{\psi}}\dot{f}(\dot{x},\pi)\neq 0\}}(xk,\pi)dkd\dot{x}d\pi\\
                 & = \iint\limits_{\widehat{G/K} G}\chi_{\{(x,\pi):G_{\psi}f(x,\pi)\neq 0\}}(x,\pi)dxd\pi\\
                & = \iint\limits_{\widehat{G} G}\chi_{\{(x,\pi):G_{\psi}f(x,\pi)\neq 0\}}(x,\pi)dxd\pi\\
                &= m_G\times m_{\widehat{G}}{\{(x,\pi):G_{\psi}f(x,\pi)\neq 0 \}}.
                \end{align*}
 So $ f = 0 $ a.e.  Hence $ \dot{f} = 0 $ a.e.    
 \end{proof}

 \begin{rem}
  A discrete  Moore group $G$, which satisfy weak QUP for Gabor transform,  is abelian. By  considering  $ f = \chi_{\{e\}}, \psi = \chi_{\{e\}},$ we see that 
  \[ \{(x,\pi):G_\psi f(x,\pi) \neq 0 \} = \{ e \} \times \widehat{G}.\]  So we have $ m\times \mu \{(x,\pi):G_\psi f(x,\pi) \neq 0 \} = \mu(\widehat{G})   \text{ which mean }\mu(\widehat{G})\geq 1\text{ as } f \neq 0$, $ \psi \neq 0. $
  Also $\mu(\widehat{G})\leq [G:G_F]^{-1}.$  Therefore $ \mu(\widehat{G}) = 1$ and $[G:G_{F}] = 1.$ Now as proved in $($\cite{kaniuth}, Lemma 3.2$)$, $G$ is abelian.

\end{rem}

\begin{prop}
  Let $G$ be a Moore group with compact identity component $ G_0$. If $G$ satisfy weak QUP for Gabor transform  then $G/G_0 $ is abelian.
 \end{prop}
 
\begin{proof}
  By structure Theorem \cite{moore}, $G/G_{0}$ is projective limit of discrete groups. Also $ G_0 $ is compact so there exist compact open normal subgroup $H_{\alpha}$ in G such that $\cap H_{\alpha} = G_0 $. Now for each $\alpha $ the weak QUP hold for $G/H_{\alpha}$ and hence $G/H_{\alpha}$ is abelian, which implies  the commutator $[G,G]\subset \cap H_{\alpha} = G_0. $
\end{proof}

\begin{prop}\label{pro3}
Let $G$ be a Lie Moore group with compact component of identity $ G_0 $ such that $G/G_0$ is abelian. If $ 0\neq f \in L^1(G) \cap L^2(G) \text{ and }  \psi $ be nonzero square integrable which is  constant on cosets of $G_0$. Then 
            \[m_G \times\mu\{(x,\pi):G_{\psi}f(x,\pi)\neq 0\} \geq 1. \]
\end{prop}

\begin{proof}
Define $ \tilde{f}(\dot{x}) = \int\limits_{G_0}f(xk)dk \text{ on } G/G_0. $ Now $\tilde{f} \in L^1(G/G_0). \text{ Since }   G_0 $ is open, therefore  $ \widehat{ G/G_0 } $ is compact and $ \mu({\widehat{G/G_0}}) < \infty. $ Moreover $ G/G_0 $ is abelian, therefore  $\tilde{f} \in L^2(G/G_0). $
Now consider,
\begin{align*}
G_{\psi}\tilde{f}(\dot{x},\pi) &= \sum\limits_{\dot{y}\in G/G_0}\tilde{f}(\dot{y})\psi(\dot{x}^{-1}\dot{y})\pi(y)^*dy\\
 &=\sum\limits_{\dot{y}\in G/G_0}\int\limits_{G_0}f(yk)\psi({x}^{-1}yk)\pi(yk)^*dydk\\
&=G_{\psi}f(x,\pi).
 \end{align*}
 Thus we have,
 \begin{align*}
& m_{G/G_0}\times \mu_{G/G_0} \{(\dot{x},\pi) \in G/G_0 \times \widehat{G/G_0} :G_{\psi}{\tilde{f}(\dot{x},\pi)} \neq 0 \}\\
  &= \sum\limits_{\dot{x} \in G/G_0}\int\limits_{\widehat{G/G_0}}\chi_{\{(\dot{x},\pi):G_{\psi}f(\dot{x},\pi)\neq 0 \}}(\dot{x},\pi)d{\dot{x}}d\pi \\
&=  \sum\limits_{\dot{x} \in G/G_0}\int\limits_{\widehat{G/G_0}}\int\limits_{G_0}\chi_{\{(xk,\pi):G_{\psi}f(xk,\pi)\neq 0 \}}(xk,\pi)dkd{\dot{x}}d\pi \\
&\leq \sum\limits_{\dot{x} \in G/G_0}\int\limits_{\widehat{G}}\int\limits_{G_0}\chi_{\{(xk,\pi):G_{\psi}f(xk,\pi)\neq 0 \}}(xk,\pi)dkd{\dot{x}}d\pi \\
&=\iint\limits_{G \widehat{G}}\chi_{\{(x,\pi):G_{\psi}f(x,\pi)\neq o \}}(x,\pi)dxd\pi\\
&= m_G \times \mu_{G}{\{(x,\pi):G_{\psi}f(x,\pi)\neq 0 \}}
 \end{align*}
 Since $G/G_0$ is abelian and $ \tilde{f},\psi $ are nonzero square integrable  functions. Hence it follows that
 \[
 m_{G/G_0}\times \mu_{G/G_0} \{(\dot{x},\pi) \in G/G_0 \times \widehat{G/G_0} :G_{\psi}{\tilde{f}(\dot{x},\pi)} \neq 0 \} \geq 1.\qedhere
 \]
 \end{proof}

\begin{lem}\label{pro2} 
Let G be a Moore group and $f, \psi \in L^2(G_F)$.

\noindent $(1)$ If g,$\phi\in L^2(G)$ is such that $g|_{G_F}=f$ and $\phi|_{G_F}=\psi$, then
\[
m_{G_F}\times\mu_{G_F}\{(x,\pi):G_{\psi}f(x,\pi)\neq 0\}\leq [G:G_F]m_G\times\mu_G\{(x,\pi):G_{\phi}g(x,\pi)\neq 0\}.
\]

\noindent $(2)$Let  $g = \tilde{f},\phi = \tilde{\psi},$ are the trivial extension of $f,\psi$ to all of G. Then
\[
m_{G_F}\times\mu_{G_F}\{(x,\pi):G_{\psi}f(x,\pi)\neq 0 \}\geq m_G\times\mu_G\{(x,\pi):G_{\phi}g(x,\pi)\neq 0 \}.
\]

\end{lem}

\begin{proof}

$(1)$ We can assume that $m_G\times\mu_G\{(x,\pi):G_{\phi}g(x,\pi)\neq 0\}  < \infty $, 
For $x \in G_F,\  g_{\phi}^x|_{G_F} = f_{\psi}^x $ then by using (\cite{kaniuth1},lemma2.2) we get that \[\mu_{G_F}{\{\pi : G_{\psi}f(x,\pi)\neq 0 \}}\leq [G:G_F]\mu_{G}\{\pi : G_{\phi}g(x,\pi) \neq 0\}
\]
i.e.$ \int\limits_{\widehat{G_F}}\chi_{\{\pi : G_{\psi}f(x,\pi)\neq 0\}}(\pi)d\pi \leq [G:G_F] \int\limits_{\widehat{G}}\chi_{\{\pi:G_{\phi}g(x,\pi)\neq 0\}}(\pi)d\pi,$

\noindent On integrating both sides we get, 
\begin{align*}
\iint\limits_{G_F \widehat{G_F}}\chi_{\{(x,\pi):G_{\psi}f(x,\pi)\neq 0\}}(x,\pi)dxd\pi &\leq[G:G_F] \iint\limits_{G_F \widehat{G}}\chi_{\{(x,\pi):G_{\phi}g(x,\pi)\neq 0\}}(x,\pi)dxd\pi\\
&\leq[G:G_F] \iint\limits_{G \widehat{G}}\chi_{\{(x,\pi):G_{\phi}g(x,\pi)\neq 0\}}(x,\pi)dxd\pi\\
\end{align*} 
or equivalently,
$m_{G_F}\times \mu_{\widehat{G_F}}\{(x,\pi):G_{\psi}f(x,\pi)\neq 0\} \leq [G:G_F]m_{G}\times \mu_{\widehat{G}}\{(x,\pi):G_{\phi}g(x,\pi)\neq 0\} $

\noindent $(2)$ For almost every  $ x \in G_F, \ g_{\phi}^x$ is trivial extension of $f_{\phi}^x.$\\ Then by (\cite{kaniuth1},Lemma 3.7)it follows that
\begin{align*}
\mu_{\widehat{G}} \{\pi: G_{\phi}g(x,\pi)\neq 0 \}&\leq \mu_{\widehat{G_F}}\{\pi : G_{\phi}f(x,\pi)\neq 0 \},\\
\int\limits_{\widehat{G}}\chi_{\{\pi:G_{\phi}g(x,\pi)\neq 0 \}}(\pi)d\pi &\leq \int\limits_{\widehat{G_F}}\chi_{\{\pi:G_{\psi}f(x,\pi)\neq 0 \}}(\pi)d\pi. \\
\end{align*}
Thus
\begin{align*}
 \iint\limits_{G_F \widehat{G}}\chi_{\{(x,\pi):G_{\phi}g(x,\pi)\neq 0 \}}(\pi)dxd\pi &\leq \iint\limits_{G_F \widehat{G_F}}\chi_{\{(x,\pi):G_{\psi}f(x,\pi)\neq 0 \}}(\pi)dxd\pi. \\
 i.e.  \iint\limits_{G \widehat{G}}\chi_{\{(x,\pi):G_{\phi}g(x,\pi)\neq 0 \}}(\pi)dxd\pi &\leq \iint\limits_{G_F \widehat{G_F}}\chi_{\{(x,\pi):G_{\psi}f(x,\pi)\neq 0 \}}(\pi)dxd\pi. \\
  m_G\times\mu_G\{(x,\pi):G_{\phi}g(x,\pi)\neq 0 \}&\leq m_{G_F}\times\mu_{G_F}\{(x,\pi):G_{\psi}f(x,\pi)\neq 0 \}.\qedhere
 \end{align*}
\end{proof}

\noindent As $G_F $ is a subgroup of finite index, so  QUP of Gabor transform holds for G if and only if it holds for $G_{F}$, but it may not be true for weak QUP of Gabor transform. Every nonabelian discrete group for which $G_F$ is abelian will serve our purpose $ e.g. (Z\rtimes \{1,-1\})$. But if $G$ has weak QUP of Gabor transform then (ii) of Lemma \ref{pro2} we can conclude that $G_F$   has weak QUP of Gabor transform. Moreover if $G_F$ has weak QUP of Gabor transform then we at least have the following,
\[
m_G\times\mu_G\{(x,\pi):G_{\phi}g(x,\pi)\neq 0 \}\geq \frac{1}{[G:G_F]}  \text{ for all nonzero } \phi, g  \in L^2(G).
\]

\noindent Next we consider bounded representation dimension group.\linebreak
 Define d(G) = sup\{dim\ $\pi :\pi \in \widehat{G} \}.$ 
\begin{lem}\label{lem3}

Let $G$ be a group with bounded representation dimension. Then
\[
m\times\mu \{(x,\pi)\in G\times\widehat{G}: G_\psi f(x,\pi)\neq 0\} \geq \frac{1}{d(G)} \text{ for } 0  \neq f, \psi \in L^2(G).
\] 
\end{lem}

\begin{proof}
 We can assume  that $ m\times\mu \{(x,\pi) \in G\times\widehat{G}: G_\psi f(x,\pi)\neq 0 \} < \infty. $
 For almost every  $ (x,\pi )\in G\times \widehat{G},  G_\psi f(x,\pi) = f_\psi ^{x}(\pi) $  and $ G_\psi f(x,\pi)^{*}  = f_\psi ^{x}(\pi)^*.$ 
Now $ \|G_\psi f(x,\pi)\|^2_{HS} = tr[G_\psi f(x,\pi)^*G_\psi f(x,\pi)]
                                         = tr[\pi({f_\psi^x})^*\pi(f_\psi^x)]$ for almost every $(x,\pi) \in G\times \widehat{G}.$
                                         
  Let $ \{ \xi_{\pi,j}1\leq j \leq d_{\pi} \}$ be an orthonormal basis of $\mathcal{H(\pi)}.$ Then,
   \begin{align*}                                      
 tr(\pi(f_\psi^x)^*\pi(f_\psi^x))&= \sum\limits_{j = 1}^{n} \langle \pi(f_\psi^x)^*\pi(f_\psi^x)\xi_{\pi,j},\xi_{\pi,j}\rangle\\
 &=\sum\limits_{j = 1}^{n}\int_G f_\psi^x(y)\langle \pi(y)\xi_{\pi,j},\pi(f_\psi^x\xi_{\pi,j})\rangle dy\\
 & = \sum\limits_{j=1}^{n}\int_G f_\psi^x(y)\overline{\langle \pi(f_\psi^x)\xi_{\pi,j},\pi(y)\xi_{\pi,j}}\rangle dy\\
  &=\sum\limits_{j=1}^{n}\iint\limits_{G,G}f_\psi^x(y) \overline{f_\psi^x(z)\langle \pi(z)\xi_{\pi,j},\pi(y)\xi_{\pi,j}\rangle}dydz\\
         &= \sum\limits_{j=1}^n \iint\limits_{G,G}f_\psi^x(y)f_\psi^x(z)\langle \pi(y)\xi_{\pi,j},\pi(z)\xi_{\pi,j}\rangle dydz\\
           &\leq d_{\pi}\| f_\psi^x \|_1^2 \leq d_G \|f\|_2^2\|\psi\|_2^2.
  \end{align*}        
 Consequently,  $tr(\pi(f_\psi^x)^*\pi(f_\psi^x))  \leq d_G\|f\|_2^2\|\psi\|_2^2.$
 Now
\begin{align*}
  \| G_\psi f\|^2 &= \int_{G\times\widehat{G}}\| G_\psi f(x,\pi)\| ^2_{(x,\pi)} d\sigma(x,\pi)\\ 
&\leq m\times\mu\{(x,\pi):G_\psi f(x,\pi)\neq 0\}\text{sup}_{x\in G,\pi \in \widehat{G}}\|G_\psi f(x,\pi)\|^2_{\text{HS}}\\
&\leq m\times\mu\{(x,\pi):G_\psi f(x,\pi)\neq 0\}d_G \|f\|_2^2\|\psi\|_2^2.\\
  \end{align*}
 or equivalently $\| f \|^2\| \psi\|^2 \leq  m\times \mu \{(x,\pi): G_\psi f(x,\pi)\neq 0 \}d_G \|f\|_2^2\|\psi\|_2^2,$
 Hence $\frac{1}{d(G)} \leq  m\times \mu \{(x,\pi): G_\psi f(x,\pi)\neq 0 \}.$
       \end{proof}

\begin{rem}
$(1)$ Weak QUP for Gabor transform always hold for abelian groups as $M = 1$. 

\noindent $(2)$ A discrete moore group satisfy weak QUP of Gabor Transform  if and only if it is abelian. 
\end{rem}

\noindent  If $G$ is a group of bounded representation dimension with connected component of identity non compact then $ m\times\mu \{(x,\pi):G_\psi f(x,\pi) \neq  0\} = \infty $ for every $ 0 \neq f, \psi \in L^2(G)$. So, the functions for which equality hold in lemma \ref{lem3} can only exist when $G_0$ is compact. 
 
 Let $N$ be a closed normal subgroup of $G$. Then  a $G-$invariant character $\gamma $ of $N$ is a continuous homomorphism from $N$ to the circle Group $ \mathbb{T}$ satisfying $\gamma(y^{-1}xy) = \gamma(x)$ for all $ x \in N$ and $y\in G$.
 
\noindent $ \widehat{G}_{\gamma} = \{\pi \in \widehat{G}:\pi(x) = \gamma(x)I_{\mathbb{H_{\pi}}} \text{ for all $ x \in N$}\} $
 
 \begin{lem}
 Let $G$ be a locally compact group of bounded representation dimension. If there exist a compact open normal subgroup $N$ of $G$ and a $G-$invariant character $\gamma$ of $N$ such that $d_\pi$ = d$($G$)$ for all most all $\pi \in \widehat{G}_{\chi}.$ Then there exist non-zero function  $  f , \psi \in L^2(G)$
 satisfying,
 \[ m\times \mu\{(x,\pi):G_\psi f(x,\pi) \neq 0 \} = \frac{1}{d(G)}.\] 
  \end{lem}
 
 \begin{proof}
 Define $ f = \overline{\gamma(x)}\chi_{N}(x) $   and $ \psi(x) = \chi_{N}(x).$ Then  \[G_\psi f(x,\pi) = \int_G f(y)\overline{\psi(x^{-1}y)}\pi(y)^* dy  = \int_N \overline{\gamma(x)}\chi_{N}(x^{-1}y)\pi(y)^* dy. \]

Thus, we have  $\{(x,\pi): G_\psi f(x,\pi) \neq 0 \} = N \times \widehat{G_{\chi}}.$ Thus by  $($\cite{kaniuth1}, Lemma 1.2$)$ it follows,
  \[ m\times \mu\{(x,\pi): G_\psi f(x,\pi) \neq 0 \} = m\times\mu (N \times \widehat{G_{\chi}}) = m(N)\mu(\widehat{G_{\chi}}) = \frac{1}{d_\pi}.\qedhere\]         
 \end{proof}
\noindent For existence of such groups see \cite{kaniuth1}.

\noindent  Now every  pair $(f,\psi)$ of nonzero functions in $L^2(G)$ such that \linebreak  $m\times \mu \{(x,\pi)\in G\times\widehat{G}: G_\psi f(x,\pi)\neq 0\} = \frac{1}{d(G)}$ is defined to be minimizing function.
 \begin{prop}
 Let $K$ be a compact normal subgroup of $G$ such that d$(\text{G/K}) = d(G)$, and suppose that there exists minimizing functions for $G/K$. Then there exist minimizing functions for $G$.
 \end{prop}
 
 \begin{proof}
 Let Haar measures on $G, K \text{ and } G/K$ be normalised so that Weil's formula hold and $m_K(K) = 1$.Then $\mu_{G/K}$ equals the measure induced from $\mu_G$ on $\widehat{G/K} \subset\widehat{G}.$ Now Let $g, \phi \in L^2(G/K) $ satisfy,
 \begin{center}
 $ m_{G/K}\times\mu_{G/K}{\{(\dot{x},\dot{\pi}):G_{\phi}g(\dot{x},\dot{\pi})\neq 0 \}} $
 \end{center}
 and define $f,\psi \in L^2(G)$ by $f(x) = g(\dot{x}), \psi(x) = \phi(\dot{x}).$ Then  by Proposition  \ref{pro1} we have,
  \[  
  m_{G}\times\mu_{G}{\{(x,\pi):G_{\psi}f(x,\pi)\neq 0 \}} =  m_{G/K}\times\mu_{G/K}{\{(\dot{x},\dot{\pi}):G_{\phi}g(\dot{x},\dot{\pi})\neq 0 \}} = 1/d(G).
  \]
 \end{proof}

 \section{Multipliers Extension of $\mathbb{T}$} 

  For a locally compact abelian group $G$ and  normalised multiplier $\omega$ on $G$, $h_{\omega}:G \to \widehat{G}$ is a continuous homomorphism where, $h_{\omega}(x)(y) = \omega(x,y)\omega(x,y)^{-1}.$ for every $x,y \in G.$
Also, $S_\omega =\{ x\in G :h_\omega(x,y)=1 , \text{ for all }y\in G\}  $ is a closed subgroup of $G$ and  $h_\omega(G) \text{ is a dense subgroup of }\widehat{G/S_\omega}\subset \widehat{G}$. For a normalised multiplier  $\omega \text{ on } G ,G(\omega)$ denote the central extension of $G$ defined by $\omega, i.e. \ G(\omega)= G\times\mathbb{T}$ with Weil topology and multiplication 
            $(x,s)(y,t) =  (xy, stw(x,y))$   for $x,y \in G $  and $s,t \in \mathbb{T}$. $G(\omega)$  is a locally compact group in which $\mathbb{T}$  is a central subgroup and $G(\omega) /\mathbb{T} \text{\ is isomorphic to G}$. .   The center of $ G(\omega)\ \text{is } S_\omega\times \mathbb{T}$.It is shown in \cite{} $\omega$ is type 1 if and only if $h_\omega$  is an open homeomorphism of G onto  $\widehat{G/S_\omega}$.

  \begin{prop}
Let $\omega$ be a type-I multiplier on an arbitrary abelian group $G$ and QUP for Gabor transform  holds for $G(\omega)$ then, $(S_{\omega})_{0}$ is non-compact. 
\end{prop}
      
\begin{proof}
 Let us assume that $(S_{\omega})_0$ is compact, which means that $ S_{\omega} $ contains a compact open subgroup $K$ and let $H = K\times\mathbb{T}$.  So $H$ is compact open subgroup of $Z = S_{\omega}\times\mathbb{T}$, the center of G($\omega$). Define $\chi\in \widehat{H} $ by\linebreak$\chi(x,t)=t,x \in K,t\in \mathbb{T}, \widehat{Z_{\chi}}=\{\alpha\in \widehat{Z} :\alpha|_H=\chi\}  \text{ and } 
     \widehat{G(\omega)_{\chi}} = \{\alpha\in\widehat{ G(\omega)}:\alpha|H \text{ is a multiple of  } \chi\}.$
   For every  $\alpha\in \widehat{Z_\chi}  \text{ there is  a unique } \pi_\alpha \in \widehat{G(\omega)} $such that $\pi_{\alpha}|Z \text{ is a multiple of } \alpha\text{ and }  \alpha\rightarrow \pi_\alpha $ is a homeomorphism between $ \widehat{Z_\chi}\text{ and } \widehat{G(\omega)}_\chi $\cite{green}.
        
\noindent Since $H$ is open in $Z$ it follows that $\widehat{G(\omega)}_{\chi}$ is compact and hence has finite Plancherel mesure.
        Define  $ f,\psi : G(\omega)\rightarrow \mathbb{C} \text{  where } f(x,t) = t\chi_{{K}}(x)$  and $ \psi(x,t) = \chi_K(x)$, then $f,\psi \in L^2(G{(\omega}))$ and we have 
          \begin{align*}
 G_\psi f(x,t,\pi) &= \int_G\int_\mathbb{T} f(y,s)\overline{\psi((x,t)^{-1}(y,s))}\pi(y,s)^{*} dyds\\
   &= \int_G\int_\mathbb{T}s\chi_K(y)\overline{\psi(x^{-1}y,\bar{t}s\bar{\omega(x^{-1},y)})}\pi(y,s)^* dyds\\
    &= \int_K\int_\mathbb{T} s\chi_K(x^{-1}y)\pi(y,s)^* dy ds.
    \end{align*}
 
\noindent      If $\ x \notin K \text{ then for all }y \in K,  x^{-1}y \notin K. $ Also  if $\pi \in \widehat{G(\omega)}/\widehat{G(\omega)_{\chi}} 
    \text{ then } \pi|_{K\times\mathbb{T}}$ is not a multiple of $\chi.$ So it follows that,
    \[\{(x,t,\pi): G_{\psi}f(x,t,\pi) \neq 0\}\subseteq K \times \mathbb{T}\times \widehat{G(\omega)_{\chi}}.
    \] 
      Thus,
       $$m_G\times \mu \{(x,t,\pi): G_{\psi}f(x,t,\pi) \neq 0\}  < \infty ,$$ 
    
 \noindent which is a contradiction to the hypothesis.    
  \end{proof}

\begin{prop}\label{pro4}

 Let $G$ be a locally compact, unimodular group of type-I and $ \omega $ is type-I multiplier on $G$.  If $ G(\omega)$ has QUP for  Gabor transform  then  $(G,\widehat{G}^\omega) $ has QUP for Gabor transform where   $\widehat{G}^\omega$ is the set of equivalence classes of irreducible  $\omega$  representation.
 
 \end{prop}
 
\begin{proof} 

Let $ f\in L^2(G)\text{ and} \ 0\neq \psi \in L^2(G)$ be a window function such that  $$  m\times \mu \{(x,\pi)\in (G,\widehat{G}^\omega): G_\psi f(x,\pi)\neq 0 \} < \infty. $$ 
\noindent Extend $ f \text{ and } \psi $ on $  G(\omega)$  as 
$ F(x,t)= tf(x)  \text{ and }  
  \Psi(x,t) = \psi(x)$ and we have  
 \begin{align}
 G_\Psi F(x,t,\pi)&= \int_{G(\omega)} F(y,s)\Psi((x,t)^{-1}(y,s))\pi(y,s)^* d\mu(y,s)\\
  &=\int_G\int_\mathbb{T}sf(y)\psi(x^{-1}y)\pi(y,s)^*dyds\label{h}
 \end{align}  
  Now   if $ \tau \in \widehat{G}^\omega $ then we can define $\rho_\tau \in 
\widehat{G(\omega)}$  by $  \rho_\tau(x,t)= t\tau(x),x\in G, t\in \mathbb{T} $

\noindent Then  $\tau\rightarrow \rho_\tau $  is a homeomorphism between $ \widehat{G}^\omega$ and
     $$ \widehat{G(\omega)_1} = \{ \rho \in \widehat{G(\omega)}: \rho|\mathbb{T} \text{  is  a  multiple  of } id_\mathbb{T}\}$$ and this maps the projective Plancherel measure on  $ \widehat{G}^\omega$ onto the restriction to $ \widehat{G(\omega)}_1 $ of the Plancherel measure on $ \widehat{G(\omega)}$.
   Now  if $\pi \in \widehat{G(\omega)}_1$, then $\pi = \pi_\tau$   for  some $  \tau \in \widehat{G}^\omega$  and $ \pi_\tau(y,s) = s\tau(y).$ 
   
 \noindent  Then(\ref{h}) becomes 
  \begin{align*}
 G_\Psi F(x,t,\pi)&= {T}sf(y)\psi(x^{-1}y)\overline{s}\tau(y)^*dsdy\\
  &= \int_G f(y)\psi(x^{-1}y)\tau(y)^*dy\\
  & =G_\psi f(x,\tau).
  \end{align*}
  
\noindent  and if $  \pi \notin \widehat{G(\omega)}_1 \text{ then }   \pi(x,t) = t^k\pi(x,1) \text{ for  some }  k\in \mathbb{Z}, k \ \neq 1$ and  hence (\ref{h})  becomes 
   \begin{align*}
   &\int_G\int_\mathbb{T}sf(y)\psi(x^{-1}y)\overline{s^k} \pi(y,1)^* ds dy\\
  &= \int_G\left(\int_{\mathbb{T}}\overline{s^{(k-1)}}dt\right)f(y)\psi(x^{-1}y)\pi(x,1)^*dy = 0.
  \end{align*}
   So,  we  have
  \begin{align*}
      \{(x,t,\pi): G_\Psi F(x,t,\pi) \neq 0\}
     &=\{(x,t,\pi_\tau): G_\Psi F(x,t,\pi_\tau) \neq 0 \}\\
   & =\{ (x,t,\pi_\tau):G_\psi f(x,\tau)\neq 0 \}\\
   & \subseteq \mathbb{T}\times \{(x,\tau): G_\psi f(x,\tau) \neq 0\},\\
  \end{align*}
  which implies that $ m\times \mu  \{(x,t,\pi): G_\Psi F(x,t,\pi) \neq 0\} <\infty.$    
 So  $F = 0$ a.e. and it follows that $f = 0$ a.e.
\end{proof}
\begin{rem}
If we take $G = \mathbb{R}^2, \text{ and } \omega((x_1,y_1),(x_2,y_2)) = e^{-\iota(y_2x_1-x_2y_1)/2},$  Then  $ \mathbb{R}^2(\omega)$  is isomorphic to three dimensional nilpotent Lie group \cite{klep}, which is Heisenberg group. Hence, $(G,G^{\omega})$ has QUP for Gabor transform. 
\end{rem}

\section*{acknowledgements}
 The first author is supported by the Junior Research Fellowship of Council of Scientific and Industrial Research, India (Ref. No:21/12/2014(ii)EU-V).

\end{document}